\documentclass[11pt]{article}
\usepackage{geometry, enumerate, amsthm, amsmath, amssymb, amscd, color}
\usepackage{hyperref}
\usepackage{courier}
\usepackage[all]{xy}

\numberwithin{equation}{section}

\newtheorem{Thm}{Theorem}[section]
\newtheorem{Prop}[Thm]{Proposition}
\newtheorem{Lem}[Thm]{Lemma}
\newtheorem{Cor}[Thm]{Corollary}
\theoremstyle{definition}\newtheorem{Def}[Thm]{Definition}
\newtheorem{Ex}[Thm]{Example}
\newtheorem{Rem}[Thm]{Remark}
\newtheorem{Rems}[Thm]{Remarks}
\theoremstyle{definition}

\newcommand{\N}{\mathbb{N}}
\newcommand{\Q}{\mathbb{Q}}
\newcommand{\Z}{\mathbb{Z}}

\newcommand{\SP}{\mathbb{P}}
\newcommand{\Int}{\textnormal{Int}}

\newcommand{\mcP}{\mathcal{P}}

\newcommand{\hZ}{\widehat{\mathbb{Z}}}
\newcommand{\Pirr}{{\mathcal P}_{\mathrm{irr}}}

\def\mm{\mathfrak m}    
\def\pp{\mathfrak p}    
\def\II{\mathfrak I}    
\def\Mm{\mathfrak M}    

\def\be{\begin{equation}}
\def\ee{\end{equation}}
\title{Polynomial overrings of $\Int(\Z)$}

\date{\today}

\author{Jean-Luc Chabert\footnote{LAMFA, UMR-CNRS 7352, University of Picardie, 33 rue Saint Leu, 80039 Amiens, France. E-mail: jean-luc.chabert@u-picardie.fr.}, Giulio Peruginelli\footnote{Department of Mathematics, University of Padova, Via Trieste, 63
35121 Padova, Italy. E-mail: gperugin@math.unipd.it.} \,\footnote{This research has been supported by the grant "Assegni Senior" of the University of Padova.}}

\begin{document}

\leftmark{\noindent   J. Commut. Algebra 8 (2016), no. 1, 1-28.\;\; \footnotesize{\href{http://dx.doi.org/10.1216/JCA-2016-8-1-1}{http://dx.doi.org/10.1216/JCA-2016-8-1-1}}.}
{\let\newpage\relax\maketitle} 


\begin{abstract}
We show that every polynomial overring of the ring $\Int(\Z)$ of polynomials which are integer-valued over $\Z$ may be considered as the ring of polynomials which are integer-valued over some subset of $\hZ$, the profinite completion of $\Z$ with respect to the fundamental system of neighbourhoods of $0$ consisting of all non-zero ideals of $\Z.$
\end{abstract}

\noindent{\small  \textbf{MSC Classification codes}: 13F20, 13B30, 13F05, 13F30}

\section*{Introduction}

The classical ring of integer-valued polynomials, namely,
$$\Int(\Z)=\{f\in\Q[X]\mid f(\Z)\subseteq\Z\},$$ 
is known to be a 2-dimensional Pr\"ufer domain (see for instance \cite[\S VI.1]{CaCh}). Thus, all the overrings of $\Int(\Z)$, that is, rings between $\Int(\Z)$ and its quotient field $\Q(X),$ are well-known {\em a priori}: they are intersections of localizations of $\Int(\Z)$ at its prime ideals, which are themselves well-known valuation domains. However, the spectrum of $\Int(\Z)$ turns out to be uncountable, so that, these intersections of localizations are not so easy to characterize. The aim of this paper is to classify the `polynomial overrings' of $\Int(\Z),$ that is, rings lying between $\Int(\Z)$ and $\Q[X]$. We first describe them as particular intersections of some families of valuation domains. 
Furthermore, we will see that the polynomial overrings of $\Int(\Z)$ may be characterized as rings of polynomials which are integer-valued over some subset of $\Z$ or, more generally, of $\hZ$, the profinite completion of $\Z$ with respect to the fundamental system of neighbourhoods of $0$ consisting of all non-zero ideals of $\Z.$

\section{Prime spectrum of $\Int(\Z)$ and localizations}\label{primespectrumIntZ}

We first recall the structure of the spectrum of $\Int(\Z)$ \cite[Prop. V.2.7]{CaCh}. A non-zero prime ideal $\mathfrak{P}$ of $\Int(\Z)$ lies over a prime ideal of $\Z,$ and hence, there are two cases:

\smallskip
 
$\bullet\; \mathfrak{P}\cap \Z=(0).$ Then $\mathfrak{P}$ is of the form
$$\mathfrak{P}=\mathfrak{P}_q=q(X)\Q[X]\cap\Int(\Z),\;\textrm{ where } q\in\Z[X] \textrm{ is irreducible}.$$
These ideals $\mathfrak{P}_q$ have height $1$ and the polynomial $q$ is uniquely determined.

\medskip

$\bullet \;\mathfrak{P}\cap \Z=p\Z$ where $p\in\SP$ (we denote by $\SP$ the set of prime numbers).
Then $\mathfrak{P}$ is of the form
$$\mathfrak{P}=\mathfrak{M}_{p,\alpha}=\{f\in\Int(\Z)\mid f(\alpha)\in p\Z_p\},\;\textrm{ where }\alpha\in\Z_p\,.$$
These ideals $\mathfrak{M}_{p,\alpha}$ are maximal ideals and the residue field of $\mathfrak{M}_{p,\alpha}$ is isomorphic to $\Z/p\Z.$ 
More precisely, $$\Z_p\ni\alpha\mapsto\mathfrak{M}_{p,\alpha}\in\textrm{Max}(\Int(\Z))$$
is a one-to-one correspondence between $\Z_p$ and the set of prime ideals of $\Int(\Z)$ lying over $p$. [Recall that $\Z_p$, the ring of $p$-adic integers, is uncountable.]

\smallskip

Moreover, given $q$ irreducible in $\Z[X]$, $p\in\SP$ and $\alpha\in\Z_p$, the following holds \cite[Prop. V.2.5]{CaCh}:
\be\label{0}\mathfrak{P}_q\subset\mathfrak{M}_{p,\alpha}\Leftrightarrow q(\alpha)=0\,.\ee 
Consequently, given an irreducible polynomial $q\in\Z[X],$ for a fixed prime $p$, there are at most finitely many ideals $\mathfrak{M}_{p,\alpha}$ containing $\mathfrak{P}_q\,;$ on the other hand, it is known that there exist infinitely many primes $p$ such that $q(X)$ has a root $\alpha$ in $\Z_p$, that is, $\mathfrak{P}_q$ is contained in infinitely many $\mathfrak{M}_{p,\alpha}$'s \cite[Prop. V.2.8]{CaCh}. In particular, the prime ideals $\mathfrak{P}_q$ are not maximal. From equivalence (\ref{0}), it follows also that the height of $\mathfrak{M}_{p,\alpha}$ is $1$ if and only if $\alpha$ is transcendental over $\Q$, it is 2 otherwise. 

\medskip

We now describe the localizations of $\Int(\Z)$ with respect to these prime ideals (see for example \cite[Prop. VI.1.9]{CaCh}). They are the following valuation domains of the field $\Q(X)$:

\smallskip

$\bullet\quad\quad \Int(\Z)_{\mathfrak{P}_q}=\Q[X]_{(q)}=\left\{\frac{f(X)}{g(X)}\in\Q(X)\mid q\nmid g\right\}\,.$

\smallskip

$\bullet\quad\quad \Int(\Z)_{\mathfrak{M}_{p,\alpha}}=V_{p,\alpha}=\{\varphi\in\Q(X) \mid \varphi(\alpha)\in\Z_p\}\,,$

\medskip

\noindent Consequently, $\Int(\Z)$ is a Pr\"ufer domain. Moreover, 
\begin{equation}\label{BqMpalpha}
V_{p,\alpha}\subset\Q[X]_{(q)}\Leftrightarrow \mathfrak{P}_q\subset\mathfrak{M}_{p,\alpha} \Leftrightarrow q(\alpha)=0\,.
\end{equation}

We are interested in the representation of $\Int(\Z)$ as an intersection of valuation overrings. For this purpose, we have to make some choices. First, we may represent $\Int(\Z)$ as the intersection of all of its valuation overrings:
$$\Int(\Z)=\bigcap_{q\in\Pirr(\Z)}\Q[X]_{(q)}\;\cap\; \bigcap_{p\in\mathbb{P}}\bigcap_{\alpha\in\Z_p}V_{p,\alpha}\,$$
where $\Pirr(\Z)$ denotes the set of irreducible polynomials of $\Z[X]$. We may look for a more optimal representation of $\Int(\Z)$. To begin with, we may discard from the above representation the valuation domains which are not minimal valuation overrings of $\Int(\Z)$, or, equivalently, the valuation domains which does not correspond to maximal ideals of $\Int(\Z)$ because $\Int(\Z)$ is a Pr\"ufer domain: 
 \begin{equation}\label{IntZ}
\Int(\Z)=\bigcap_{p\in\mathbb{P}}\bigcap_{\alpha\in\Z_p}V_{p,\alpha}\,.
\end{equation}
The above intersection in (\ref{IntZ}) is uncountable and it is far from being irredundant. Recall that, given a domain $D$ with quotient field $K$, and a family of valuation overrings $\Lambda=\{V_{\lambda}\}$ of $D$ (that is, $D\subseteq V_{\lambda}\subset K$) such that $D=\bigcap_{\lambda}V_{\lambda}$, the representation $D=\bigcap_{\lambda}V_{\lambda}$ is said {\em irredundant} if no $V_{\lambda}$ is superfluous, that is, for each $\lambda$, $D$ is strictly contained in the intersection of the member of $\Lambda$ distinct from $V_{\lambda}$ (\cite{GilmHeinz}). For the domain $\Int(\Z)$, there are suitable countable intersections as shown, for instance, by the following equality:
\be\label{aaa}\Int(\Z)=\bigcap_{p\in\mathbb{P}}\bigcap_{\alpha\in\Z}V_{p,\alpha}\,.\ee
The fact that every rational function on the right-hand side of equality~(\ref{aaa}), that is, that every $\varphi\in\Q(X)$ such that $\varphi(\Z)\subseteq\Z$ is a polynomial follows from the observation that a rational function which takes integral values on infinitely many integers is a polynomial (see~\cite[VIII.2 (93)]{bib:polyand} or \cite[Prop. X.1.1]{CaCh}).

So, every $V_{p,\alpha}$, $\alpha\in \Z_p\setminus \Z$, $p\in\SP$, in the representation (\ref{IntZ}) is superfluous; actually, we will show that, for each $p\in\SP$ and $\alpha\in\Z_p$, every $V_{p,\alpha}$ in the above representation is superfluous (Corollary \ref{superfluousVpa}). However, there is no irredundant representation of $\Int(\Z)$ as an intersection of valuation overrings because there is no subset of $\Z$ which is minimal among the subsets of $\Z$ which are dense in $\Z$ for every $p$-adic topology (see Corollary \ref{irredundantlocal} and Remark \ref{Remark52}). Thus, in the sequel, the only representations that we will consider as `canonical' will be the intersections of all the minimal valuation overrings as in (\ref{IntZ}).

\medskip

After some generalities about the overrings of $\Int(\Z)$ in Section 2, we consider the representations of the overrings of $\Int(\Z_{(p)}),$ where $p$ is a fixed prime number and $\Z_{(p)}$ denotes the localization of $\Z$ at $p\Z$ in Section 3, as intersections of valuation domains (Proposition \ref{representations}) and then, as rings of integer-valued polynomials on a subset of $\Z_p$ (Theorem \ref{containmentoverrings}); in particular, we show that there is a one-to-one correspondence between the set of polynomial overrings of $\Int(\Z_{(p)})$ and the closed subsets of $\Z_p$. In order to globalize these results, we study in Section 4 the valuation overrings of an intersection of valuation domains, characterizing those which are superfluous (Corollary \ref{superfluousVpa} and Theorem \ref{vqunder}). Finally, the polynomial overrings of $\Int(\Z)$ are described in Section 5 by their representations as intersection of valuation overrings (Proposition \ref{representationsglobal}), and in Section 6 with an interpretation as integer-valued polynomials on a subset of the ring $\hZ$ (Theorem \ref{globalintpoly}).


\section{Generalities about overrings of $\Int(\Z)$}
We are interested in rings $R$ which are overrings of $\Int(\Z)$, that is,
\begin{equation}\label{1}
\Int(\Z)\subseteq R\subseteq \Q(X),
\end{equation}
and, in particular, by the {\em polynomial overrings} of $\Int(\Z),$ that is, the rings $R$ which are  contained in $\Q[X]$. 

Since $\Int(\Z)$ is a Pr\"ufer domain, we first recall the following fundamental result of \cite{GilmOverring} (see also \cite[Theorem 26.1]{Gilm}) concerning overrings $D'$ of a Pr\"ufer domain $D$, that is, rings $D'$ such that $D\subseteq D'\subseteq K$ where $K$ denotes the quotient field of $D$. 

\begin{Prop}\label{overringPrufer}
Let $D'$ be an overring of a Pr\"ufer domain $D,$ and let $\mathcal{S}_{D'}$ be the set of prime ideals $\pp$ of $D$ such that $\pp D'\subsetneq D'.$ Then
\begin{itemize}
\item[$(i)$] If $\pp'$ is a prime ideal of $D'$ and $\pp=\pp'\cap D$, then $D_{\pp}=D'_{\pp'}$ and $\pp'=\pp D_{\pp}\cap D'$.  Therefore $D'$ is Pr\"ufer.
\item[$(ii)$] If $\pp$ is a non-zero prime ideal of $D$, then $\pp$ is in $\mathcal{S}_{D'}$ if and only if $D_{\pp}\supseteq D'$. Moreover, $D'=\bigcap_{\pp\in\mathcal{S}_{D'}}D_{\pp}$.
\item[$(iii)$] Every ideal $\II'$ of $D'$ is an extended ideal, that is, $ \II'=(\II'\cap D)\,D'.$ 
\item[$(iv)$] The spectrum of $D'$ is $\{\pp D'\mid \pp\in\mathcal{S}_{D'}\}$.
\end{itemize}
\end{Prop}

In view of the previous proposition, we will use the following terminology: a prime ideal $\pp$ of $D$ is said to {\em survive} in $D'$ if its extension $\pp D'$ in $D'$ is a proper ideal (that is, $\pp D'\subsetneq D'$, in which case $\pp D'$ is a prime ideal of $D'$ by the above result) and $\pp$ is said to be {\em lost} in $D'$ otherwise (that is, if $\pp D'=D'$). In particular, every overring $D'$  of a Pr\"ufer domain $D$ is equal to the intersection of the localizations of $D$ at those prime ideals $\pp$ of $D$ which survive in $D'$.

\begin{Ex}\label{Q[X]}
Clearly,
$$\Q[X]=\bigcap_{q\in\mcP_{\textnormal{irr}}}\Int(\Z)_{\mathfrak{P}_q}=\bigcap_{q\in\mcP_{\textnormal{irr}}}\Q[X]_{(q)}\,,$$
where $\mcP_{\textnormal{irr}}=\mcP_{\textnormal{irr}}(\Z)$ is the set of irreducible polynomials in $\Z[X]$. By \cite[Remark 1.12]{GilmHeinz}, this representation of $\Q[X]$ is irredundant, since $\Q[X]$ is a Dedekind domain and the set of maximal ideals of $\Q[X]$ is in one-to-one correspondence with $\mcP_{\textnormal{irr}}$, namely $\mcP_{\textnormal{irr}}\ni q\mapsto q(X)\Q[X]$. 

Consequently, for a polynomial overring $R$, each prime ideal $\mathfrak{P}_q$ of $\Int(\Z)$  must survive in $R$ since it survives in $\Q[X],$ and we have $$\mathfrak{P}_q\, R=q(X)\Q[X]\cap R.$$
\end{Ex}

Since we want to describe explicitly $R$ in terms of those prime ideals of the spectrum of $\Int(\Z)$ which survive in $R$, we are mostly interested in the other prime ideals, those lying over a prime. They are called {\em unitary} prime ideals because they contain nonzero constants.

\smallskip

The following result of Gilmer and Heinzer is of fundamental importance in order to decide whether an ideal $\pp$ of $\Int(\Z)$ survives or not in some intersection of valuation overrings of $\Int(\Z)$.

\begin{Prop}\emph{(\cite[Prop. 1.4]{GilmHeinz})}\label{GH}
Let $D$ be a Pr\"ufer domain and let $\{\pp\}\cup\{\pp_{\alpha}\}_{\alpha\in\Lambda}$ be a family of prime ideals of $D$. Then $D_{\pp}\supseteq\bigcap_{\alpha\in\Lambda}D_{\pp_{\alpha}}$ if and only if, for every finitely generated ideal $\II\subseteq \pp$, there exists $\alpha\in\Lambda$ such that $\II\subseteq \pp_{\alpha}$.
\end{Prop}

\begin{Cor}\label{minimal}
If $D_{\pp}$ is not a minimal valuation overring of the Pr\"ufer domain $D,$ then $D_{\pp}$ is superfluous in each representation of $D$ as an intersection of valuation overings in which $D_{\pp}$ appears.
\end{Cor}

\begin{proof}
(See also~\cite[Lemma 1.6]{GilmHeinz}.) Let $\bigcap_{\alpha\in\Lambda}D_{\pp_\alpha}$be any representation of $D$, let $\alpha_0\in\Lambda,$ and assume that $D_{\pp_{\alpha_0}}$ is not a superfluous element in this representation. By Proposition~\ref{GH}, there exists a finitely generated ideal $\II\subseteq \pp_{\alpha_0}$ such that $\II\not\subseteq \pp_\alpha$ for every $\alpha\in\Lambda\setminus\{\alpha_0\}.$ Let $\mm$ be a maximal ideal of $D$ containing $\pp_{\alpha_0}$ and let $x$ be any element of $\mm$. Since $D=\bigcap_{\alpha\in\Lambda}D_{\pp_\alpha}$ and $\II+(x)\not\subseteq \pp_\alpha$ for $\alpha\not=\alpha_0$, necessarily $\II+(x)\subseteq \pp_{\alpha_0}$. Finally, $\pp_{\alpha_0}=\mm$ is maximal, which is equivalent to the fact that $D_{\pp_{\alpha_0}}$ is a minimal valuation overring of $D$.
\end{proof}

\begin{Rems}
(1) The converse of the previous corollary may be false: there are minimal valuation overrings which may be superfluous (cf. Example~\ref{exbelow} below). 

(2) We have to take care that there is another notion of minimality which depends on the representation that we consider: a valuation domain which is minimal with respect to the elements of some representation of $D$ is not necessarily minimal with respect to another representation (and in particular, with respect to all the valuation overrings of $D$). For instance, let $p\in\SP$, $\alpha_n\in\Z$ ($n\geq 0$) and $q\in\Pirr(\Z)$ such that $q(\alpha_0)=0$ and $\lim_{n\to+\infty} v_p(\alpha_n-\alpha_0)=+\infty.$ Let $V_q=\Q[X]_{(q)}$. Then, we have:
\be\label{exmini}D\doteqdot(\cap_{n\geq 0}V_{p,\alpha_n})\cap V_q = \cap_{n\geq 0}V_{p,\alpha_n} = \cap_{n>0}V_{p,\alpha_n} = 
(\cap_{n>0}V_{p,\alpha_n})\cap V_q \,.\ee
The first equality follows from the fact that $V_q\supset V_{p,\alpha_0}$ and the second equality from the fact that $\alpha_0=\lim_{n\to\infty} \alpha_n$ in $\Z_p$ (see Lemma \ref{lemmalopmore}). The valuation domain $V_q$ is not minimal with respect to the elements of the first representation, while it is for the last one. 

(3) Obviously, a valuation domain which is not minimal with respect to some representation is superfluous for this representation, but Corollary~\ref{minimal} says something stronger since a minimal valuation overring of $D$ which appears in some representation of $D$ is {\em a fortiori} minimal for this representation. In the last representation of $D$ given in (\ref{exmini}), $V_q$ is superfluous although it is minimal for this representation, but we could be sure that it is superfluous because it is not a minimal overring of $D$ as shown by the first representation.
\end{Rems}

Thus, we emphasize that when we speak of a minimal valuation overring of $D$ it is always a valuation domain which is minimal with respect to the family of {\em all} the valuation overrings of $D$.

Another important example is the localization of $\Int(\Z)$ with respect to a prime $p\in\Z$.

\begin{Ex}\label{ExamplelocalizzIntZ}
For every fixed prime $p$, we have
$$\Int\left(\Z_{(p)}\right)=\Int(\Z)_{(p)}$$
where $\Int(\Z)_{(p)}$ is the localization of the $\Z$-module $\Int(\Z)$ at $p\Z$, namely, $\Int(\Z)_{(p)}=\left\{\frac{f(X)}{s}\mid f\in\Int(\Z), s\in\Z\setminus p\Z\right\}$ (see \cite[Thm I.2.3]{CaCh}). Consequently, the prime ideals of $\Int(\Z)$ which survive in $\Int(\Z_{(p)})$ are the non-unitary ideals $\mathfrak{P}_q$ and the unitary ideals $\mathfrak{M}_{p,\alpha}$ lying over the prime $p$. By a slight abuse of notation, we still denote the corresponding extended ideals in $\Int(\Z_{(p)})$ by $\mathfrak{P}_q$ and $\mathfrak{M}_{p,\alpha}$, respectively. Then we have:
$$\Int(\Z_{(p)})=\bigcap_{q\in\mcP_{\textnormal{irr}} }\Q[X]_{(q)}\cap\bigcap_{\alpha\in\Z_p}V_{p,\alpha} = \bigcap_{q\in\mcP_{\textnormal{irr}}}\Q[X]_{(q)}\cap\bigcap_{\alpha\in\Z}V_{p,\alpha}\,.$$
But, in this local case, an ideal $\mathfrak{P}_q$ may be maximal in $\Int(\Z_{(p)}): \mathfrak{P}_q$ is maximal if and only if $q(X)$ has no root in $\Z_p$ (\cite[Prop. V.2.5]{CaCh}). Therefore, if $\mcP_{\textnormal{irr}}^{\Z_p}$ denotes the set of irreducible polynomials over $\Z$ which have no roots in $\Z_p$, we have the following representation of $\Int(\Z_{(p)})$ as the intersection of all its minimal valuation overrings (which correspond to the maximal ideals of $\Int(\Z_{(p)})$ ):
\be\Int(\Z_{(p)})=\bigcap_{q\in\mcP_{\textnormal{irr}}^{\Z_p}}\Q[X]_{(q)}\cap\bigcap_{\alpha\in\Z_p}V_{p,\alpha}\,.\ee
\end{Ex}

\begin{Rem}\label{PZpnonempty}
It is not difficult to see that ${\mathcal P}_{\mathrm{irr}}^{\Z_p}$ is non-empty: let $g\in\Z_p[X]$ be a monic irreducible polynomial of degree $d\geq 2$. By a corollary of Krasner's lemma (see for instance \cite[Chapter V, Proposition 5.9]{Nark}), every monic polynomial $q\in \Z_p[X]$ of degree $d$ which is sufficiently close to $g(X)$ with respect to the $p$-adic valuation is also irreducible over $\Z_p[X]$. Clearly, we may choose such a polynomial $q(X)$ with coefficients in $\Z$. Then, in particular, $q(X)$ is irreducible in $\Z[X]$ and has no roots in $\Z_p$.
\end{Rem}
\vskip0.4cm
If we localize each ring of (2.1) at $p$ (that is, with respect to the multiplicative set $\Z \setminus p\Z$), since $\Int(\Z)$ is well-behaved under localization as seen in Example \ref{ExamplelocalizzIntZ}, we get
\begin{equation}\label{Rp}
\Int(\Z_{(p)})\subseteq R_{(p)}\subseteq\Q(X)
\end{equation}
where $R_{(p)}=\{\frac{f(X)}{n}\mid f\in R,n\in\Z\setminus p\Z\}$. If $R$ is a polynomial overring of $\Int(\Z)$, then $R_{(p)}$ is a polynomial overring of $\Int(\Z_{(p)})$, that is, $R_{(p)}\subseteq\Q[X]$. Clearly, we have
\begin{equation}\label{R=Rp}
R=\bigcap_{p\in\mathbb{P}}R_{(p)}
\end{equation}
Hence, in order to make our work easier, we fix a prime $p$ and we continue our discussion for an overring $R$  of $\Int(\Z_{(p)})$.


\section{Polynomial overrings of $\Int\left(\Z_{(p)}\right)$}

In this section, $p$ denotes a fixed prime number and we consider overrings of $\Int(\Z_{(p)})$ that is, rings $R$ such that 
$$\Int\left(\Z_{(p)}\right)\subseteq R\subseteq \Q(X).$$

\noindent{\bf Notation}. For every overring $R$ of $\Int(\Z_{(p)}),$ we consider the following subsets:
\begin{enumerate} 
\item A subset of the ring $\Z_p$ of $p$-adic integers
\begin{equation}\label{OmegaRp}
Z_p(R)\doteqdot\{\alpha\in\Z_p\mid\mathfrak{M}_{p,\alpha}R\subsetneq R\}
\end{equation}
\item For every $\alpha\in\Z_p$ which is not the pole of some element of $R$, the following subring of the field $\Q_p$ of $p$-adic numbers
$$R(\alpha)\doteqdot\{f(\alpha)\mid f\in R\}\subseteq\Q_p\,.$$

\end{enumerate}

Note that $Z_p({R})$ indexes the set of maximal unitary ideals of $\Int(\Z_{(p)})$ which survive in $R$ under extension, and that $R(\alpha)$ is always defined for polynomial overrings of $\Int(\Z_{(p)})$.
For instance, if $R=\Int(\Z_{(p)})$, then $Z_p(R)=\Z_p$ and, for every $\alpha\in Z_p(R)\cap\Q, R(\alpha)=\Z_{(p)}$, since $\Z_{(p)}[X]\subset R_{(p)}$ and $R(\alpha)\subseteq \Z_p\cap\Q$. 
The following easy proposition characterizes the set $Z_p(R)$ for any overring $R$.

\begin{Prop}\label{CharacterizOmegaRp}
Let $R$ be an overring of $\Int(\Z_{(p)})$ and $\alpha\in\Z_p$. Then
\be\label{Z(R)}\alpha\in Z_p(R)\Leftrightarrow R\subseteq V_{p,\alpha}\Leftrightarrow R(\alpha)\subseteq\Z_p\,.\ee
Moreover, the subset $Z_p(R)$ is closed in $\Z_p$ for the $p$-adic topology.
\end{Prop}

\begin{proof}
The first equivalence follows from Proposition~\ref{overringPrufer}.  The second equivalence is straightforward from the definitions of $V_{p,\alpha}$ and $R(\alpha)$.
Concerning the last assertion, note that, for each $f\in R$, by continuity of $f,$ the subset $\{\alpha\in \Z_p\mid f(\alpha)\in \Z_p\}$ is closed in $\Z_p$. Then, we just have to remark that:

\centerline{$Z_p(R)=\bigcap_{f\in R}\;\{\alpha\in \Z_p\mid f(\alpha)\in \Z_p\}.$}\end{proof}

\begin{Cor}\label{nonunitaryRp}
Under extension, a prime ideal $\mathfrak{P}_q$ of $\Int(\Z_{(p)})$ is maximal in $R$ if and only if $q(X)$ has no roots in $Z_p(R)$.
\end{Cor}

\begin{proof}
If $\mathfrak{P}_q$ does not become maximal in $R$ under extension, then $\mathfrak{P}_qR$ is  strictly contained is some prime ideal $\mathfrak{Q}$ of $R$. By Proposition~\ref{overringPrufer}, $\mathfrak{Q}$ must be equal to the extension of some prime ideal of $\Int(\Z_{(p)})$, which must be a maximal ideal $\mathfrak{M}_{p,\alpha}$ containing $\mathfrak{P}_q$, or equivalently, $V_{p,\alpha}\subset \Q[X]_{(q)}$. In particular, $\alpha\in Z_p(R)$ and $q(\alpha)=0$, by (\ref{BqMpalpha}). The converse is clear.
\end{proof}

\subsection{Polynomial overrings of $\Int\left(\Z_{(p)}\right)$  as intersections of valuation domains}
Now we consider different representations of a polynomial overring $R$ as intersections of valuation overrings of $\Int(\Z_{(p)})$.

\begin{Prop}\label{representations}
Let $p$ be a prime and let $R$ be any polynomial overring of $\Int(\Z_{(p)})$. We have the following representations of $R$ as an intersection of valuation overrings.
\begin{enumerate}
\item[$(i)$] The intersection of all the valuation overrings:
\be\label{allR} R=\bigcap_{q\,\in\,\Pirr}\Q[X]_{(q)}\;\cap\; \bigcap_{\alpha\,\in\, Z_p(R)}V_{p,\alpha}\ee
where $\Pirr$ denotes the set of irreducible polynomials of $\Z[X]$, and $Z_p(R)$ is defined by $Z_p(R)\doteqdot\{\alpha\in\Z_p\mid\mathfrak{M}_{p,\alpha}R\subsetneq R\}$. 
\item[$(ii)$] The intersection of all the minimal valuation overrings:
\begin{equation}\label{representationRp}
R=\bigcap_{q\,\in\,\Pirr^{Z_p(R)}}\Q[X]_{(q)}\;\cap\; \bigcap_{\alpha\,\in \, Z_p(R)}V_{p,\alpha}
\end{equation}
where $\Pirr^{Z_p(R)}$ denotes the subset of ${\mathcal P}_{\mathrm{irr}}$ formed by those polynomials which have no roots in $Z_p(R)$.
\item[$(iii)$] For every ${\mathcal P}\subseteq\Pirr$ and every $E\subseteq Z_p(R)$, the following intersection of valuation overrings of $R$:
\begin{equation}\label{necp}
R_{{\mathcal P}, E}=\bigcap_{q\,\in\,{\mathcal P}}\Q[X]_{(q)}\;\cap\; \bigcap_{\alpha\,\in\, E}V_{p,\alpha}\,
\end{equation}
is equal to $R$ if and only if ${\mathcal P}\supseteq \Pirr^{Z_p(R)}$ and $E$ is $p$-adically dense in $Z_p(R)$.
\end{enumerate}
\end{Prop}

\begin{proof}
Example~\ref{Q[X]} and equivalences~(\ref{Z(R)}) show clearly that the valuation overrings of $R$ are exactly those which appear in the right-hand side of equality~(\ref{allR}). The equality follows from the fact that $R$ is an overring of a Pr\"ufer domain, and hence, it is a Pr\"ufer domain, equal to the intersection of all its valuation overrings. Thus, $(i)$ is proved.

The minimal valuation overrings of $R$ correspond to the valuation overrings whose center is a maximal ideal of $R$. Assertion $(ii)$ is then a consequence of Corollary~\ref{nonunitaryRp}.

By equality (\ref{allR}), $R$ is contained in any ring of the form $R_{{\mathcal P},\, E}.$ By continuity of the rational functions, if $\beta\in\Z_p$ is the limit of a sequence $\{\alpha_n\}_{n\geq 0}$ of elements of $E$, then $V_{p,\beta}\supset\bigcap_{n\,\in\, \N}V_{p,\alpha_n}\supset \bigcap_{\alpha\,\in\, E}V_{p,\alpha}$. As a consequence, if $E$ is dense in $Z_p(R)$, then $\bigcap_{\alpha\,\in\, E}V_{p,\alpha}= \bigcap_{\alpha\,\in\,Z_p(R)}V_{p,\alpha}$, and hence, once more by equality (\ref{representationRp}), $R_{{\mathcal P},\, E}=R.$ 

Let us prove now the converse assertion of $(iii)$. Assume first that ${\mathcal P}\not\supset \Pirr^{Z_p(R)}$. Then, there exists $r\in\Pirr\setminus{\mathcal P}$ without any root in $Z_p(R)$. Let $m=\sup\,\{v_p(r(\alpha))\mid \alpha\in Z_p(R)\}$. Since $Z_p(R)$ is closed, $m$ is finite since otherwise there would exists a sequence $\{\alpha_n\}_{n\geq 0}$ of elements of $Z_p(R)$ such that $v_p(r(\alpha_n))\geq n,$ and by compactness of $Z_p(R)$, there would exist a subsequence which converges to an element $\beta,$ which then would be a root of $r(X)$ in $Z_p(R)$. Consider now the rational function $\varphi(X)=\frac{p^m}{r(X)}$. For every  $\alpha\in Z_p(R)$,  $v_p(r(\alpha))\leq m$, and hence, $\varphi\in V_{p,\alpha}$. Consequently, $\varphi\in\bigcap_{q\in\mathcal{P}}\Q[X]_{(q)} \cap \,\bigcap_{\alpha\in Z_p(R)}V_{p,\alpha},$ while clearly $\varphi\notin \Q[X]_{(r)}$. Thus, $R\subsetneq R_{{\mathcal P}, E}$. 

Assume now that $E$ is not $p$-adically dense in $Z_p(R)$. It remains to prove that again we have a strict containment: $R\subsetneq R_{{\mathcal P}, E}$. For this, it is enough to prove that:
$$R\subsetneq \left(\bigcap_{q\,\in\,\Pirr}\Q[X]_{(q)}\right)\;\cap\; \left(\bigcap_{\alpha\,\in\, E}V_{p,\alpha}\right)=\{f(X)\in\Q[X]\mid f(E)\subseteq \Z_p\}.$$
This strict containment is a clear consequence of Proposition~\ref{topclospolclos} below. 
\end{proof}

\begin{Rem}\label{Remark}
By Remark \ref{PZpnonempty} and by the fact that $\Pirr^{\Z_p}\subseteq\Pirr^{Z_p(R)}$ for each overring $R$ of $\Int(\Z_{(p)})$, it follows that $\Pirr^{Z_p(R)}$ is always non-empty. Note though, that the complement of $\Pirr^{Z_p(R)}$ may be empty, for example if $Z_p(R)$ is formed by elements of $\Z_p$ which are transcendental over $\Q$.
\end{Rem}

As for $\Int(\Z)$ or for $\Int\left(\Z_{(p)}\right)$, an overring $R$ does not have in general an irredundant representation as intersection of valuation overrings. There does exist an irredundant representation in some particular cases, as the next result shows.

\begin{Cor}\label{irredundantlocal}
A polynomial overring $R$ of $\Int(\Z_{(p)})$ admits an irredundant representation if and only if $Z_p(R)$ contains a $p$-adically dense subset $E$ formed by isolated points.
\end{Cor}

\begin{proof}
Assume that $R_{\mathcal{P},E}$ is an irredundant representation of $R$. By Proposition~\ref{representations}$(iii)$, $\mathcal{P}={\mathcal P}_{\mathrm{irr}}^{Z_p(R)}$ and $E$ is dense in $Z_p(R)$. Moreover, for each $\alpha_0\in E$, $R=R_{\mathcal{P},E}\subsetneq R_{\mathcal{P},E\setminus\{\alpha_0\}}$, thus the topological closure of $E\setminus\{\alpha_0\}$ is strictly contained in that of $E$, which means that $\alpha_0$ is isolated in $E$. The reverse implication is obvious still by Proposition~\ref{representations}$(iii)$.
\end{proof}

For instance, we can consider $E$ to be equal to the set of distinct elements of a convergent sequence $\{\alpha_n\}_{n\geq 0}$ with limit $\alpha$, so that $Z_p(R)=E\cup\{\alpha\}$.

\subsection{Polynomial overrings of $\Int\left(\Z_{(p)}\right)$ as integer-valued polynomials rings}\label{polynomialclosure}

Contrarily to equality (\ref{representationRp}), equality (\ref{allR}) shows that a polynomial overring $R$ depends only on $Z_p(R)$. In order to describe how a polynomial overring $R$ of $\Int(\Z_{(p)})$ is characterized by its associated set $Z_p(R),$ we recall the following definition (for example, see \cite{PerFinite,PerWer}).

\begin{Def}\label{defint}
For every  subset $E$ of $\Z_p$, the ring formed by the polynomials of $\Q[X]$ whose values on $E$ are $p$-integers is denoted by:
$$\Int_{\Q}(E,\Z_p)\doteqdot\{f\in\Q[X]\mid f(E)\subset\Z_p\}.$$
In particular, for $E=\Z_p$, we set $\Int_{\Q}(\Z_p)\doteqdot\Int_{\Q}(\Z_p,\Z_p)$.
\end{Def}

By definition (or by convention) we set $\Int_{\Q}(\emptyset,\Z_p)=\Q[X]$ (after all, any polynomial is integer-valued over the empty-set). Note also that $\Int_{\Q}(E,\Z_p)=\Q[X]\cap\Int(E,\Z_p)$ where $\Int(E,\Z_p)=\{f\in\Q_p[X]\mid f(E)\subseteq\Z_p\}\,.$
The following equality follows from a continuity-density argument:
\begin{equation}\label{IntZp}
\Int(\Z_{(p)})=\Int_{\Q}(\Z_p)
\end{equation}

\begin{Prop}\label{RpIntval}
Let $R$ be a polynomial overring of $\Int(\Z_{(p)})$ and let $Z_p(R)=\{\alpha\in\Z_p\mid \mathfrak{M}_{p,\alpha}R\subsetneq R\}$. Then
$$R=\Int_{\Q}(Z_p(R),\Z_p).$$
\end{Prop}

\begin{proof}
The containment $R\subseteq\Int_{\Q}(Z_p(R),\Z_p)$ follows from Proposition  \ref{CharacterizOmegaRp}: $\alpha\in Z_p(R)$ if and only $R(\alpha)\subseteq\Z_p$. Thus, we have the chain of inclusions:
$$\Int(\Z_{(p)})\subseteq R\subseteq\Int_{\Q}(Z_p(R),\Z_p)\subseteq\Q[X]\,.$$
In order to prove the converse containment, it is sufficient to show that each prime ideal of $\Int(\Z_{(p)})$ which survives in $R$ also survives in $\Int_{\Q}(Z_p(R),\Z_p).$ In fact, since we are dealing with Pr\"ufer domains, if a prime ideal $\mathfrak{P}$ of $\Int(\Z_{(p)})$ is such that $\mathfrak{P}R\subsetneq R$, then $\mathfrak{P}R$ is a prime ideal of $R$ and these extensions comprise the  whole spectrum of $R$ by Proposition \ref{overringPrufer}, iv). We then use the well-known fact that an integral domain is equal to the intersection of the localizations at its own prime ideals.

For what we have already said, all the prime non-unitary ideals survive in both rings since they survive in $\Q[X]$. Let $\mathfrak{M}_{p,\alpha}$ be a maximal unitary ideal which survives in $R$. By definition of $Z_p(R)$, $\alpha\in  Z_p(R)$. Now, $\mathfrak{M}_{p,\alpha}$ survives in $\Int_{\Q}(Z_p(R),\Z_p)$ if and only if $\Int_{\Q}(Z_p(R),\Z_p)$ is contained in $V_{p,\alpha}$, that is, each polynomial of $\Int_{\Q}(Z_p(R),\Z_p)$ is integer-valued on $\alpha$. Since $\alpha\in Z_p(R)$, the conclusion follows.
\end{proof}

In particular, from Proposition \ref{RpIntval}, we have a complete characterization of the family $\mathcal{R}_p$ of polynomial overrings of $\Int(\Z_{(p)}):$ 
\begin{Cor}
If $\mathcal{F}(\Z_p)$ denote the family of closed subsets of $\Z_p,$ then
$$\mathcal{R}_p=\{\Int_{\Q}(F,\Z_p)\mid F\in\mathcal{F}(\Z_p)\}\,.$$
\end{Cor}

Proposition \ref{RpIntval} says how $R$ is characterized by the closed subset $Z_p(R)\subseteq\Z_p$. In order to prove that for different closed subsets of $\Z_p$ we get different polynomial overrings of $\Int(\Z_{(p)})$, we recall the notion of polynomial closure introduced by Gilmer~\cite{gilmer} and McQuillan~\cite{McQ}.

\begin{Def}\label{polyclo}
For any subset $E\subseteq\Z_p$ ($E$ is not necessarily closed), the $p$-{\em polynomial closure} of $E$ is the largest subset $\overline{E}$ of $\Z_p$ (containing $E$) such that
$$\Int_{\Q}(E,\Z_p)=\Int_{\Q}(\overline{E},\Z_p)\,.$$
\end{Def}
Equivalently,
$$\overline{E}=\{\alpha\in\Z_p\mid \Int_{\Q}(E,\Z_p)(\alpha)\subset\Z_p\}=\{\alpha\in\Z_p\mid \Int_{\Q}(E,\Z_p)\subset V_{p,\alpha}\}=Z_p(\Int_{\Q}(E,\Z_p)),$$
where the last equality follows by Proposition \ref{CharacterizOmegaRp}.

\begin{Prop}\label{topclospolclos}
For any subset $E\subseteq\Z_p$,  the following subsets are equal:

$(i)$ the $p$-polynomial closure of $E,$

$(ii)$ the $p$-adic topological closure of $E,$

$(iii)$ $Z_p(\Int_{\Q}(E,\Z_p)).$
\end{Prop}

For the equivalence between the polynomial closure and the topological closure see for instance \cite[Thm IV.1.15]{CaCh}.
The next theorem shows that the closed subsets of $\Z_p$ are in one-to-one correspondence with the polynomial overrings of $\Int(\Z_{(p)})$.

\begin{Thm}\label{containmentoverrings}
Let $\mathcal{R}_p$ be the set of polynomial overrings of $\Int(\Z_{(p)})$ and let $\mathcal{F}(\Z_p)$ be the family of closed subsets of $\Z_p$. The following maps which reverse the containments are inverse to each other:
$$\varphi_p:\mathcal{R}_p\ni R\mapsto Z_p(R)\in\mathcal{F}(\Z_p)\quad\textrm{ and }\quad \psi_p:\mathcal{F}(\Z_p)\ni F\mapsto\Int_{\Q}(F,\Z_p)\in\mathcal{R}_p$$
\end{Thm}

\begin{proof}
By  Proposition \ref{RpIntval}, $\psi_p\circ\varphi_p=id_{\mathcal{R}_p}$. Now we consider $\varphi_p\circ\psi_p:$ for every $F\in\mathcal{F}(\Z_p)$, one has $\varphi_p(\psi_p(F))=Z_p(\Int(F,\Z_p))=\{\alpha\in\Z_p\mid \forall f\in\Int(F,\Z_p)\;f(\alpha)\in\Z_p\}=F$ by Proposition~\ref{topclospolclos} since $F$ is assumed to be closed. Consequently, $\varphi_p\circ\psi_p=id_{\mathcal{F}(\Z_p)}$.
\end{proof}

We end this section with the characterization of minimal ring extensions of the family $\mathcal{R}_p.$ Recall that $R_1\subsetneq R_2\in\mathcal{R}_p$ forms a minimal ring extension if there is no ring in between $R_1$ and $R_2.$

\begin{Prop}\label{minimalextension}
Let $R=\Int_{\Q}(F,\Z_p)$ where $F=Z_p(R)$ is a closed subset of $\Z_p$. There is a bijection between the minimal ring extensions of $R$ in $\mathcal{R}_p$ and the subset $F_0$ formed by the isolated points of $F,$ which is given by:
$$F_0\ni\alpha\mapsto\Int_{\Q}(F\setminus\{\alpha\},\Z_p)$$
\end{Prop}

We stress that we are interested only in {\em polynomial} ring extensions of $R$, that is, elements of the family $\mathcal{R}_p$. Note that the proposition says that $R$ has no minimal ring extension in $\mathcal{R}_p$ if and only if $Z_p(R)$ has no isolated points.

\begin{proof}
Let $S\in{\mathcal R}_p$ be a proper extension of $R$. Then, by Theorem \ref{containmentoverrings}, $S=\Int_{\Q}(E,\Z_p)$ where $E=Z_p(S)$ is a closed subset strictly contained in $F$. For every $\alpha\in F\setminus E,$ the subset $E\cup\{\alpha\}$ is closed and the ring $T=\Int_{\Q}(E\cup\{\alpha\},\Z_p)$ satisfies $R\subseteq T\subsetneq S$ since $E\subsetneq E\cup\{\alpha\}\subseteq F.$ 

Therefore, the extension $R\subsetneq S$ is minimal if and only if there is no closed subset $G$ such that $E\subsetneq G\subsetneq F$. Consequently, if the extension $R\subsetneq S$ is minimal, then necessarily $F=E\cup\{\alpha\}$. The fact that $E$ is closed in $E\cup\{\alpha\}=F$ implies that $\alpha$ is isolated in $F$. Conversely, if $\alpha\in F$ is isolated in $F$, then $F\setminus\{\alpha\}$ is closed in $F$ and clearly there is no closed subset $G$ properly lying between $F\setminus\{\alpha\}$ and $F$. Thus, we may conclude that $S$ is a minimal extension of $R$ if and only if $S=\Int_{\Q}(F\setminus\{\alpha\},\Z_p)$ where $\alpha\in F$ is an isolated point.
If $\alpha\not=\alpha'$ are two distinct isolated points of $F$, then by Theorem \ref{containmentoverrings} the corresponding minimal ring extensions of $R$ are distinct, because $F\setminus\{\alpha\}\not= F\setminus\{\alpha'\}$.
\end{proof}

\section{Valuation overrings of an intersection of valuation domains}

The aim of this section is to characterize whether a valuation overring of $\Int(\Z)$ as described in section \ref{primespectrumIntZ} contains a given intersection of valuation overrings of $\Int(\Z)$. We will apply the obtained results to describe the representations of every polynomial overring of $\Int(\Z)$ as intersections of valuation domains. In order to do this, we will use extensively Proposition~\ref{GH}. To ease the notation, we set $V_q=\Q[X]_{(q)}$, for $q\in\Pirr$. Moreover, since now we are going to consider arbitrary intersections of unitary valuation domains for different $p\in\SP$, we generalize the notation $
R_{{\mathcal P},\, E_p}$ used in formula (\ref{necp}) in the following way: if $\mathcal{P}\subseteq\Pirr$ and if, for each $p\in\SP$, $E_p\subseteq\Z_p,$ then we set
$$R_{{\mathcal P},\, (E_p)_{p\in\SP}}=\bigcap_{q\,\in\,{\mathcal P}}V_q\;\cap\; \bigcap_{p\in\SP}\;\bigcap_{\alpha\,\in\, E_p}V_{p,\alpha}\,
$$
If the subset $E_p$ of $\Z_p$ is empty for some $p\in\SP,$ then the corresponding intersection $\bigcap_{\alpha\in E_p}V_{p,\alpha}$ is set to be equal to $\Q(X)$. We consider a similar convention for the set of non-unitary valuation overrings $V_q$ if $\mathcal{P}=\emptyset$. In particular, if $E_p$ is empty for all $p\in\SP$ except $p_0$, then the intersection corresponds to the ring $R_{\mathcal{P},E_{p_0}}.$ 

We want to determine which are the valuation overrings of a ring $R_{{\mathcal P},\, (E_p)_{p\in\SP}}$ as above. We distinguish the case of a unitary valuation overring $V_{p,\alpha}$ (whose center is a unitary prime ideal of $\Int(\Z)$) from a non-unitary valuation overring $V_q$ (whose center is non-unitary).

\subsection{Unitary valuation overrings}

We begin to determine unitary valuation overrings of an arbitary intersection of $V_{p,\alpha}$ for a fixed prime $p$, and  possibly some non-unitary valuation domains $V_q$'s. We remark first that, given a subset $E$ of $\Z_p$, if $V_{p_0,\alpha_0}$ is an overring of $\cap_{\alpha\in E}V_{p,\alpha}$, where $p_0\in\mathbb{P}$ and $\alpha_0\in\Z_{p_0}$, then $p_0=p$. In fact, if that were not true, then $\frac{1}{p_0}$, which is in $\cap_{\alpha\in E}V_{p,\alpha}$ would also belongs to $V_{p_0,\alpha_0}$, which is a contradiction. Therefore, we can just consider valuation overrings which lie above the same prime $p$. 

The next result is an obvious consequence of Proposition \ref{topclospolclos} (see also \cite[Lemma 26]{LopMore}; although Section 5 of \cite{LopMore} is entitled `Overrings of $\Int(\Z)$', the author's point of view is quite different from ours). 

\begin{Lem}\label{lemmalopmore}
Let $p\in\SP,$ $E\subseteq\Z_p,$ $\mathcal{P}\subseteq\Pirr$ and $\alpha_0\in\Z_{p}$. The following assertions are equivalent:
\begin{enumerate}
\item[$(i)$] $R_{{\mathcal P}, E}\subseteq V_{p,\alpha_0}\,,$
\item[$(ii)$] $\Int_{\Q}(E,\Z_p)\subseteq V_{p,\alpha_0}\,,$
\item[$(iii)$] $\alpha_0$ belongs to the topological closure $\overline{E}$ of $E$ in $\Z_p$.
\end{enumerate}
In particular, $R_{{\mathcal P}, E}=R_{{\mathcal P}, \overline E}$ and $Z_p(R_{{\mathcal P}, E})=\overline{E}$.
\end{Lem}

\begin{proof}
$(i)\rightarrow (ii)$: $\Int_{\Q}(E,\Z_p)$ is contained in $R_{{\mathcal P}, E}$. 

$(ii)\leftrightarrow (iii)$: $(ii)$ means that, for every $f\in  \Int_{\Q}(E,\Z_p)$, $f(\alpha_0)\in\Z_p$, that is, $\alpha_0$ belongs to the $p$-polynomial closure of $E,$ thus we may conclude with Proposition~\ref{topclospolclos}. 

$(ii)\rightarrow (i)$:
Assume that  $V_{p,\alpha_0}$ is an overring of $R=\Int_{\Q}(E,\Z_p)$. 
We use Proposition~\ref{GH} to get the claim. Let $I\subset R$ be a finitely generated ideal contained in $\mathfrak{M}_{p,\alpha_0}$ and let $J=I+(p)$. Since $J$ is not contained in any non-unitary prime ideal of $R$, then it follows from (\ref{representationRp}) that $J$ is contained in some unitary prime ideal $\mathfrak{M}_{p,\alpha}$ of $R$ where $\alpha\in E$. In particular, $I$ is contained in this ideal $\mathfrak{M}_{p,\alpha}$ and we conclude that $V_{p,\alpha_0}\supseteq\bigcap_{\alpha\in E}V_{p,\alpha}\supseteq R_{{\mathcal P}, E}$.

The last claims follow immediately.
\end{proof}

\begin{Lem}
For each $p\in\SP$, let $E_p\subseteq\Z_p.$ Let $p_0\in\SP$ and $\alpha_0\in\Z_{p_0}$. Then 
$$\bigcap_{p\in\mathbb{P}}\bigcap_{\alpha\in E_p}V_{p,\alpha}\subset V_{p_0,\alpha_0}\Leftrightarrow \bigcap_{\alpha\in E_{p_0}}V_{p_0,\alpha}\subset V_{p_0,\alpha_0}
$$
\end{Lem}

\begin{proof}
One implication is obvious. Conversely, assume that $V_{p_0,\alpha_0}$ is an overring of the intersection $\bigcap_{p\in\mathbb{P}}\bigcap_{\alpha\in E_p}V_{p,\alpha}$. Let $I$ be a finitely generated ideal contained in $\mathfrak{M}_{p_0,\alpha_0}$ and let $J=I+(p_0)$. Since for all $p\not=p_0$ and for all $\alpha\in E_p,$ we have $J\not\subset\mathfrak{M}_{p,\alpha},$ it follows that $J\subseteq\mathfrak{M}_{p_0,\alpha}$ for some $\alpha\in E_{p_0}$. In particular, $I\subseteq\mathfrak{M}_{p_0,\alpha}$. By Proposition~\ref{GH}, we may conclude.
\end{proof}

Both previous lemmas lead to the following proposition.

\begin{Prop}\label{unitaryvaloverrings}
For each $p\in\SP,$ let $E_p\subseteq\Z_p$. Let $p_0\in \SP,$ let $\alpha_0\in\Z_p,$ and let $\mathcal{P}$ be any subset of $\Pirr.$ Then the following conditions are equivalent:
\begin{itemize}
\item[$(i)$] $\bigcap_{p\in\mathbb{P}}\bigcap_{\alpha\in E_p}V_{p,\alpha}\subset V_{p_0,\alpha_0}$.
\item[$(i')$] $R_{{\mathcal P},\, (E_p)_{p\in\SP}}\subset V_{p_0,\alpha_0}$.
\item[$(ii)$] $\bigcap_{\alpha\in E_{p_0}}V_{p_0,\alpha}\subset V_{p_0,\alpha_0}$.
\item[$(ii')$] $R_{\mathcal{P},E_{p_0}}\subset V_{p_0,\alpha_0}$.
\item[$(iii)$] $\alpha_0$ is in the topological closure of $E_{p_0}$ in $\Z_{p_0}.$ 
\end{itemize}
\end{Prop}
\medskip
\begin{Cor}\label{superfluousVpa}
Let $\mathcal{P}\subseteq\Pirr$ and, for each $p\in\SP$, let $E_p\subseteq\Z_p$. Then, $V_{p_0,\alpha_0}$ where $p_0\in\SP$ and $\alpha_0\in E_{p_0}$ is not a superfluous valuation overring of $R_{{\mathcal P},\, (E_p)_{p\in\SP}}$ if and only if $\alpha_0$ is an isolated point of $E_{p_0}$.
\end{Cor}
\begin{proof}
$V_{p_0,\alpha_0}$ is not a superfluous valuation overring of $R_{{\mathcal P},\, (E_p)_{p\in\SP}}$ if and only if the intersection of the valuation domains of the family $\{V_q | q\in\mathcal{P}\}\cup\{V_{p,\alpha}|\alpha\in E_p,p\in\SP\}\setminus\{V_{p_0,\alpha_0}\}$ is not contained in $V_{p_0,\alpha_0}$. By Proposition \ref{unitaryvaloverrings}, this condition is equivalent to $\alpha_0\notin \overline{E_{p_0}\setminus \{\alpha_0\}}$, that is, $\alpha_0$ is an isolated point of $E_{p_0}$. 
\end{proof}

\subsection{Non-unitary valuation overrings}

Now we consider the case of a non-unitary valuation domain $V_q=\Q[X]_{(q)}$, $q\in\Pirr$, containing an arbitrary intersection of unitary and non-unitary valuation domains.

\begin{Lem}\label{Vq0overringVq}
Let $\mathcal{P}\subset\Pirr$ and $q_0\in\Pirr$. Then
$$\bigcap_{q\in\mathcal{P}}V_{q}\subseteq V_{q_0}\Leftrightarrow q_0\in\mathcal{P}$$
\end{Lem}
\begin{proof}
One direction is obvious. Conversely, suppose $V_{q_0}$ is an overring of the intersection of the $V_q$'s, $q\in\mathcal{P}$. If $q_0\notin\mathcal{P}$, we have a contradiction since 
$$\frac{1}{q_0}\in\bigcap_{q\in\mathcal{P}}V_{q}\;\textrm{ and }\;\frac{1}{q_0}\notin V_{q_0}\,.$$
\end{proof}

\begin{Thm}\label{vqunder}
Let $q\in\Pirr$ and for each $p\in\mathbb{P}$, let $F_p\subseteq\Z_p$ be a $($possibly empty$)$ closed set of $p$-adic integers. Then
$$\bigcap_{p\in\mathbb{P}}\bigcap_{\alpha\in F_p}V_{p,\alpha}\subset V_q\Leftrightarrow\; \exists\; (\alpha_p)\in\prod_{p\in\mathbb{P}}F_p\textnormal{ such that }\sum_{p\in\mathbb{P}}v_p(q(\alpha_p))=+\infty$$
\end{Thm}
Note that the latter condition means that: either there exist $p\in\mathbb{P}$ and $\alpha_p\in F_p$ such that $q(\alpha_p)=0,$ or there exist infinitely many primes $p_n\in\mathbb{P}$ and some $\alpha_{p_n}\in F_{p_n}$ such that $v_{p_n}(q(\alpha_{p_n}))\geq 1$. Example~\ref{exbelow} below shows that the latter condition can really occur.

\begin{proof}
Since $\mathfrak{P}_q=\{qf\in\Int(\Z)\mid f\in\Q[X]\}$ and $\Q[X]$ is countable, we may fix a sequence $\{f_n\}_{n\geq 0}$ of polynomials in $\Q[X]$ such that the $qf_n$'s generate $\mathfrak{P}_q$. We also consider the set
$$\mathbb{P}_{q}\doteqdot\{p\in\mathbb{P}\mid \exists\alpha_p\in F_p \textnormal{ such that }q\in \mathfrak{M}_{p,\alpha_p}\}.$$

Assume first that there exists $(\alpha_p)\in\prod_{p\in\mathbb{P}} F_p$ such that $\sum_{p\in\mathbb{P}}v_p(q(\alpha_p))=\infty$. If for some prime $p$ and some $\alpha\in F_p$, $q(\alpha)=0$, then $V_q\supset V_{p,\alpha},$ and hence, $V_q\supset\bigcap_{p\in\mathbb{P}}\bigcap_{\alpha\in F_p}V_{p,\alpha}$. Suppose that, for each $p\in\SP$, $q(X)$ has no roots in $F_p$. It follows that the set 
$\mathbb{P}_{q}$ is infinite. 
Let $I\subseteq \mathfrak{P}_q$ be any finitely generated ideal. There exists $n$ such that $I\subseteq(qf_1,\ldots,qf_n)$. Since, for almost all $p\in\SP$, the polynomials $f_1,\ldots,f_n$ are in $\Z_{(p)}[X]$, there exists $p\in\SP_q$ such that $f_1,\ldots,f_n\in\Z_{(p)}[X],$ and hence, for the above $\alpha_p\in F_p$, $v_p(q(\alpha_p)f_j(\alpha_p))\geq v_p(q(\alpha_p))>0$, for $1\leq j\leq n$. Consequently, $I\subseteq \mathfrak{M}_{p,\alpha_p}$, which shows by Proposition~\ref{GH} that $V_q\supset\bigcap_{p\in\mathbb{P}}\bigcap_{\alpha\in F_p}V_{p,\alpha}$.

Conversely, assume that $\mathbb{P}_q=\{p_1,\ldots,p_s\}$ and $m_{i}=\sup\{v_{p_i}(q(\alpha))\mid\alpha\in F_{p_i}\}<\infty$ for $i=1,\ldots,s$ ($\Leftrightarrow q(\alpha)\not=0$, for each $\alpha\in F_{p_i}$, $i=1,\ldots,s$ since $F_{p_i}$ is closed). Then, consider the rational function:
$$\varphi(X)=\frac{\prod_{i}^s p_i^{m_i}}{q(X)}$$
For every $p_i\in\SP_q$ and every $\alpha_{p_i}\in F_{p_i}$,  $v_p(q(\alpha_{p_i}))\leq m_i$, and hence, $\varphi\in V_{p_i,\alpha_{p_i}}$. For every $p\in\SP\setminus\mathbb{P}_q$ and every $\alpha_p\in F_p$, $v_p(q(\alpha_p))=0$, and hence, $\varphi\in V_{p,\alpha_p}$. Consequently, $\varphi\in\bigcap_{p\in\mathbb{P}}\bigcap_{\alpha_p\in F_p}V_{p,\alpha_p},$
while clearly $\varphi\notin V_q$.
\end{proof}
\medskip
\begin{Ex}\label{exbelow}
This example shows that a minimal non-unitary valuation overring of some ring $R_{{\mathcal P},\, (F_p)_{p\in\SP}}$ can be superfluous. Let $q(X)=X$. Suppose $\mathfrak{P}_q=\bigcup_{n\in\N}I_n$ where $\{I_n\}_{n\geq 0}$ is an increasing sequence of ideals, each of them generated by $Xf_1(X),\ldots,Xf_n(X)$, for some $f_i\in\Q[X]$. Let $p_n$ be the $n$-th prime. For each $n\in\N$, there exists $a_n\in\N$ large enough such that $I_n\subset\mathfrak{M}_{p_n,p_n^{a_n}}$, exactly by the same argument of the above proof. Then, by Proposition~\ref{GH}, $V_q$ is an overring of $\bigcap_{n\in\N}V_{p_n,\,p_n^{a_n}}$, even though, by (\ref{BqMpalpha}), $V_{p_n,\,p_n^{a_n}}\not\subset V_q$ for each $n\in\N$. Hence, $V_q$ is a minimal overring of $\bigcap_{n\in\N}V_{p_n,\,p_n^{a_n}}\cap V_q$ which is superfluous. Or, if we want to consider integer-valued polynomials, let $E=\cup_{n\in\N}\{p_n^{a_n}\}$, then we have
$$\Int(E,\Z)=\bigcap_{q\,\in\,{\mathcal P}_{\mathrm{irr}}}\Q[X]_{(q)}\;\cap\; \bigcap_{p\in\SP}\;\bigcap_{n\,\in\, \N}V_{p,\,p_n^{a_n}}$$
where the minimal valuation overring $V_X$ of $\Int(E,\Z)$ is superfluous.
\end{Ex}
\medskip

\begin{Rems}\label{Remarks} Let $\mathcal{P}\subseteq\Pirr, q_0\in\Pirr,$ and, for each $p\in\SP$, let $F_p$ be a closed subset of $\Z_p.$   

(1) Theorem \ref{vqunder} may be generalized to an arbitrary intersection of unitary and non-unitary valuation domains:
$$R_{{\mathcal P},\, (F_p)_{p\in\SP}}\subset V_{q_0}\Leftrightarrow\; \textnormal{ either }q_0\in\mathcal{P}\textnormal{ or }\exists\; (\alpha_p)\in\prod_{p\in\mathbb{P}}F_p\textnormal{ such that }\sum_{p\in\mathbb{P}}v_p(q_0(\alpha_p))=+\infty\;$$
In fact, if $V_{q_0}$ is an overring of $R_{{\mathcal P},\, (F_p)_{p\in\SP}}$ and there is no $(\alpha_p)\in\prod_{p\in\mathbb{P}}F_p$ such that $\sum_{p\in\mathbb{P}}v_p(q_0(\alpha_p))=+\infty$, then, by Theorem \ref{vqunder}, $V_{q_0}$ is not an overring of $\bigcap_{p\in\mathbb{P}}\bigcap_{\alpha\in F_p}V_{p,\alpha}$. Hence, by the techniques of Proposition \ref{GH}, $V_{q_0}$ is easily seen to be an overring of $\bigcap_{q\in\mathcal{P}}V_q$, and so, by Lemma \ref{Vq0overringVq}, $q_0\in\mathcal{P}$, as wanted. Conversely, by Theorem \ref{vqunder}, each condition on the right-hand side implies that $V_{q_0}$ is an overring of $R_{{\mathcal P},\, (F_p)_{p\in\SP}}$.

(2) If $F_p$ is an empty set for all but finitely many primes $\{p_1,\ldots,p_n\}$ (for example, overrings $R_{\mathcal{P},F_p}$ of $\Int(\Z_{(p)})$), then $V_q$ is an overring of $R_{{\mathcal P},\, (F_p)_{p\in\SP}}$ if and only if $q(X)$ has a root is some $F_{p_i}$, $i=1,\ldots,n$, or $q\in\mathcal{P}$. Therefore, a minimal non-unitary valuation overring $V_q$ of $R_{{\mathcal P},\, (F_{p_i})_{i=1,\ldots,n}}$ is not superfluous. 

\end{Rems}

\section{Polynomial overrings of $\Int\left(\Z\right)$ as intersections of valuation domains}

We consider now a polynomial overring $R$ of $\Int(\Z)$. Analogously to the previous case of overrings of $\Int(\Z_{(p)})$ we consider the subset $Z_p(R)$.

\medskip

\noindent{\bf Notation}. For every ring $R$ such that $\Int(\Z)\subseteq R\subseteq \Q[X]$ and every $p\in\SP$, let $Z_p(R)$ be the  following subset of $\Z_p:$
\be Z_p(R)\doteqdot\{\alpha\in\Z_p\mid \Mm_{p,\alpha}R\subsetneq R\}\,. \ee
We already introduced $Z_p(R)$ in (\ref{OmegaRp}) for polynomial overrings of $\Int(\Z_{(p)})$. Fortunately both notations agree to each other since clearly $Z_p(R)=Z_p(R_{(p)}):$
\begin{equation}\label{ZpRZpRp}
\alpha\in Z_p(R)\Leftrightarrow R\subseteq V_{p,\alpha}\Leftrightarrow R_{(p)}\subseteq V_{p,\alpha}\Leftrightarrow \alpha\in Z_p(R_{(p)})\,.
\end{equation}

Analogously to Proposition~\ref{representations}, we consider now the representations of $R$ as an  intersection of valuation overrings.

\begin{Prop}\label{representationsglobal}
Let $R$ be any polynomial overring of $\Int(\Z)$. We have the following representations of $R$ as an intersection of valuation overrings.
\begin{enumerate}
\item[$(i)$] The intersection of all the valuation overrings:
\begin{equation}\label{allRglobal}
R=\bigcap_{q\,\in\,{\mathcal P}_{\mathrm{irr}}}\Q[X]_{(q)}\;\cap\; \bigcap_{p\in\SP}\;\bigcap_{\alpha\,\in\, Z_p(R)}V_{p,\alpha}
\end{equation}
\item[$(ii)$] The intersection of all the minimal valuation overrings:
\begin{equation}\label{representationRpglobal}
R=\bigcap_{q\,\in\,{\mathcal P}_{\mathrm{irr}}^{Z(R)}}\Q[X]_{(q)}\;\cap\; \bigcap_{p\in\SP}\;\bigcap_{\alpha\,\in \, Z_p(R)}V_{p,\alpha}
\end{equation}
where ${\mathcal P}_{\mathrm{irr}}^{Z(R)}$ denotes the set of irreducible polynomials of $\Z[X]$ which have no roots in $Z_p(R)$ whatever $p\in\SP$.
\item[$(iii)$] For every ${\mathcal P}\subseteq{\mathcal P}_{\mathrm{irr}}$ and every $E_p\subseteq Z_p(R)\;(p\in\SP),$ the following intersection of valuation overrings of $R$
\begin{equation}\label{nec} 
R_{{\mathcal P},\, (E_p)_{p\in\SP}}=\bigcap_{q\,\in\,{\mathcal P}}\Q[X]_{(q)}\;\cap\; \bigcap_{p\in\SP}\;\bigcap_{\alpha\,\in\, E_p}V_{p,\alpha}\,
\end{equation}
is equal to $R$ if and only if 
\begin{enumerate}
\item[$(a)$] ${\mathcal P}\supseteq{\mathcal P}_{\mathrm{irr}}^{Z_0(R)}$
 where ${\mathcal P}_{\mathrm{irr}}^{Z_0(R)}$ is formed by the irreducible polynomials $q$ of $\Z[X]$ such that, for every $p\in\SP$, $q$ has no root in $Z_p(R)$, and there do not exist two infinite sequences $\{p_j\}_{j\in \N}$ and $\{\alpha_j\}_{j\in\N}$ where $p_i\in \SP$, $\alpha_j\in Z_{p_j}(R)$, and $v_{p_j}(q(\alpha_j))>0,$  
\item[$(b)$] for every $p\in\SP$, $E_p$ is $p$-adically dense in $Z_p(R)$.
\end{enumerate}
\end{enumerate}
\end{Prop}

\begin{proof}
Formula (\ref{allRglobal}) is clearly a consequence of Proposition~\ref{representations}\,(i).
Analogously to Formula (\ref{representationRp}), Formula (\ref{representationRpglobal}) follows from the globalization of Corollary~\ref{nonunitaryRp}: a prime ideal $\mathfrak{P}_q$ of $\Int(\Z)$ is maximal in $R$ if and only if, for each $p\in\mathbb{P}$,  $q(X)$ has no roots in $Z_p(R)$. 

It remains to prove assertion (iii). If (a) and (b) hold then, by Proposition~\ref{unitaryvaloverrings} and Theorem~\ref{vqunder}, the right-hand side of Formula~(\ref{nec}) is equal to the right-hand side of Formula~(\ref{allRglobal}), and hence to $R$. 

Assume now that (a) does not hold: there exists $r\in{\mathcal P}_{\mathrm{irr}}^{Z_0(R)}\setminus{\mathcal P}.$ 
Since $E_p\subseteq Z_p(R)$ for every $p\in\SP$ and $r\notin {\mathcal P}$, it follows from Theorem~\ref{vqunder} that $V_r=\Q[X]_{(r)}\not\supseteq R_{{\mathcal P},\, (E_p)_{p\in\SP}};$ in particular, $\Q[X]\not\supseteq R_{{\mathcal P},\, (E_p)_{p\in\SP}},$ and hence, $R\subsetneq R_{{\mathcal P},\, (E_p)_{p\in\SP}}$. 

Finally, assume that (b) does not hold: there is some $p_0\in\SP$ such that $E_{p_0}$ is not $p_0$-adically dense in $Z_{p_0}(R)$, in other words, there is some $\alpha_0\in Z_{p_0}(R)$ which is not in the topological closure of $E_{p_0}$. By Proposition~\ref{unitaryvaloverrings},
$\bigcap_{p\in\mathbb{P}}\bigcap_{\alpha\in E_p}V_{p,\alpha}\cap \Q[X]\not\subseteq V_{p_0,\alpha_0},$ and hence, once more, $R\subsetneq R_{{\mathcal P},\, (E_p)_{p\in\SP}}$.
\end{proof}

\begin{Rem}\label{Remark52}
We can generalize Corollary \ref{irredundantlocal} to overrings of $\Int(\Z)$ in the following way: a polynomial overring $R$ of $\Int(\Z)$ admits an irredundant representation if and only if, for each $p\in\SP$,  $Z_p(R)$ contains a $p$-adically dense subset formed by isolated points.
\end{Rem}


\section{Polynomial overrings of $\Int\left(\Z\right)$ as integer-valued polynomial rings over subsets of $\hZ$}

In this section we give another point of view about polynomial overrings $R$ of $\Int(\Z)$, in order to represent them as rings of integer-valued polynomials. We know that $R=\cap_{p\in\SP}R_{(p)}$ and that $R_{(p)}=\Int_{\Q}(Z_p(R),\Z_p)$. Consequently, $R$ is equal to an intersection of different integer-valued polynomial rings as $p$ runs through the set of prime numbers:
\begin{equation}\label{inter}
R=\bigcap_{p\in\SP}\Int_{\Q}(Z_p(R),\Z_p)\,.
\end{equation}
However, it seems to be more convenient to consider all the $p$-adic completions $\Z_p$ at the same time. Classically, the way to do that is via the ring of {\em finite adeles} ${\mathcal A_f(\Q)}$ (`finite' refers to the fact we forget the Archimedean absolute value). A finite {\em adele} is an element $\underline{\alpha}=(\alpha_p)_p$ of the product $\prod_{p\in\mathbb{P}} \Q_p$ such that for all but finitely many $p$'s, $\alpha_p$ belongs to $\Z_p$ (for instance, see~\cite[Section 6.2, p. 286]{Nark}). 

Note that $\Q$ embeds diagonally into $\prod_{p\in\mathbb{P}} \Q_p$ and its image is in $\mathcal{A}_f(\Q),$ actually $\Q$ embeds into the group of units of $\mathcal{A}_f(\Q):$
recall that this group, denoted by $\mathcal I_f(\Q)$ and called finite ideles, is formed by the elements $\underline{\alpha}=(\alpha_p)_p\in\prod_{p\in\SP}\Q_p^*$ such that $v_p(a_p)=0$ for all but finitely many $p$. Given $\underline{\alpha}=(\alpha_p)_p\in\mathcal A_f(\Q)$ and $f\in\Q[X]$, we have clearly
\begin{equation*}
f(\underline{\alpha})=(f(\alpha_p))_p\in\mathcal A_f(\Q)\subset\prod_{p\in\mathbb{P}} \Q_p,
\end{equation*}
that is, every polynomial with rational coefficients maps an adele into an adele. For this reason, the ring of integer-valued polynomials over the ring of finite adeles is trivial: 
$$\Q[X]=\Int_{\Q}(\mathcal A_f(\Q))=\{f\in\Q[X] \mid f(\underline{\alpha})\in\mathcal{A}_f(\Q),\forall \underline{\alpha}\in\mathcal{A}_f(\Q)\}.$$
However, note that $\mathcal A_f(\Q)$ contains as a subring the product $\prod_{p\in\SP}\Z_p$, which is isomorphic to $\hZ$, the profinite completion of $\Z$ with respect to the fundamental system of neighbourhoods of $0$ consisting of all the non-zero ideals of $\Z$.


Given $f\in\Q[X]$ and $\underline{\alpha}\in\mathcal{A}_f(\Q)$,  we say that $f$ is {\em integer-valued} at $\underline{\alpha}$ if $f(\underline{\alpha})=(f(\alpha_p))_p\in\hZ= \prod_{p}\Z_p$. 
Then, analogously to Definition~\ref{defint}, we introduce the following:

\begin{Def}
For every subset $\underline{E}$ of $\widehat{\Z}$, the ring of integer-valued polynomials on $\underline{E}$ is
$$\Int_{\Q}(\underline{E},\widehat{\Z})=\{f\in\Q[X] \mid f(\underline{\alpha})\in\widehat{\Z},\;\forall\underline{\alpha}\in\underline{E}\}\,.$$
\end{Def}

\noindent{\bf Notation}.
For each polynomial overring $R$ of $\Int(\Z)$, we consider the following set of finite adeles:
$$\underline{Z}_R\doteqdot\prod_{p} Z_p(R)\subseteq\prod_{p}\Z_p=\hZ$$
Clearly, $\underline{Z}_R=\{(\alpha_p)_p \in \hZ \mid \mathfrak{M}_{p,\alpha_p}R\subsetneq R, \forall p\in\SP \}$. With the previous notation, Equality~(\ref{inter}) may then be written:
\be\label{Z-R} R=\Int_{\Q}(\underline{Z}_R,\hZ)\,,\ee
which means that every polynomial overring $R$ of $\Int(\Z)$ may be considered as the ring formed by polynomials which are integer-valued over a subset of $\hZ$. 
Note that, since for each $p\in\mathbb{P}$, $Z_p(R)$ is a closed subset of $\Z_p$, and hence, is compact, the subset $\underline{Z}_R$ is also compact in $\hZ$ where $\hZ=\prod_{p\in\SP}\Z_p$ is endowed with the product topology. The following theorem is the globalized version of Theorem~\ref{containmentoverrings}.

\begin{Thm}\label{globalintpoly}
Let $\mathcal{R}$ be the set of polynomial overrings of $\Int(\Z)$ and let $\mathcal{F}(\hZ)$ be the family of compact subsets of $\hZ$ of the form $\prod_{p\in\SP}F_p$ where $F_p$ is a closed subset of $\Z_p$. The following maps which reverse the containments are inverse to each other:
$$\varphi:\mathcal{R}\ni R\mapsto \underline{Z}_R=\prod_{p\in\SP}Z_p(R)\in\mathcal{F}(\hZ)\quad\textrm{ and }\quad \psi:\mathcal{F}(\hZ)\ni\underline{F}\mapsto\Int_{\Q}(\underline{F},\hZ)\in\mathcal{R}\,.$$
\end{Thm}

\begin{proof}
By Equality (\ref{Z-R}), $\psi\circ\varphi=id_{\mathcal{R}}$. Consider now $\varphi\circ\psi:$ 
for every $\underline{F}=\prod_pF_p\in\mathcal{F}(\Z_p)$, one has
$\varphi(\psi(\underline{F}))=\underline{Z}_{\Int_{\Q}(\underline{F},\Z_p)}=\{(\alpha_p)_p\in\prod_p\Z_p\mid \forall f\in\Int(\underline{F},\hZ),\;\forall p\in\SP,\; f(\alpha_p)\in\Z_p\}=\{(\alpha_p)_p\in\prod_p\Z_p\mid \Int(\underline{F},\hZ) \subseteq V_{p,\alpha_p},\forall p\in\SP\}$ which is equal to $\underline{F}$ by Proposition~\ref{unitaryvaloverrings}.
\end{proof}
\medskip
\begin{Rem}\label{polynomial closure in hZ}
Let $\underline{F}$ be a generic compact subset of $\hZ$ and consider the following ring of integer-valued polynomials:
$$R=\Int_{\Q}(\underline{F},\hZ).$$
For each $p\in\mathbb{P}$, let $\pi_p:\hZ\to\Z_p$ be the canonical projection. Then, for each $f\in\Int_{\Q}(\underline{F},\hZ)$ and for each $\alpha_p\in\pi_p(\underline{F})$, $p\in\SP$,  we have $f(\alpha_p)\in\Z_p.$ Consequently, $f\in\Int_{\Q}(\pi_p(\underline{F}),\Z_p)$. Therefore,
$$R\subseteq\bigcap_p \Int_{\Q}(\pi_{p}(\underline{F}),\Z_p)=\Int_{\Q}(\prod_p \pi_{p}(\underline{F}), \prod_p \Z_p)\subseteq \Int_{\Q}(\underline{F},\hZ)=R $$
since $\underline{F}\subseteq\prod_p \pi_{p}(\underline{F})$.
Finally,
$$\Int_{\Q}(\underline{F},\hZ)=\Int_{\Q}(\prod_p \pi_{p}(\underline{F}),\hZ).$$
Since the projections $\pi_p$ are closed maps, each $\pi_p(\underline{F})$ is a closed subset of $\Z_p$. Therefore, by Theorem \ref{globalintpoly}, we have proved that $\underline{Z}_{\Int_{\Q}(\underline{F},\hZ)}=\prod_p \pi_{p}(\underline{F})$, which is an element of $\mathcal{F}(\hZ)$. In other words, $\prod_p \pi_{p}(\underline{F})$ is precisely equal to the subset of $\hZ$ of those $\underline{\alpha}$ such that $f(\underline{\alpha})\in\hZ$, for each $f\in \Int_{\Q}(\underline{F},\hZ)$. Generalizing the terminology of section \ref{polynomialclosure}, one could say that the {\em polynomial closure} of $\underline{F}\subseteq\hZ$ is the compact subset $\prod_p \pi_{p}(\underline{F})$.
\end{Rem}
\medskip
\begin{Rem}
Let $E\subseteq\Z$ be an infinite subset. We denote by $\widehat{E}$ the topological closure of $E$ in $\hZ=\prod_{p\in\mathbb{P}}\Z_p$, by $E_p$ the topological closure of $E$ in $\Z_p$, for each prime $p$ and by $\underline{E}$ the direct product $\prod_{p\in\mathbb{P}}E_p\subseteq\hZ$. By Remark \ref{polynomial closure in hZ}, $\underline{E}$ is the polynomial closure of $E$ in $\hZ$.  It is easy to see that 
$$\widehat{E}\subseteq\underline{E}$$
since the canonical embedding of $E$ into $\hZ$ is contained in $\prod_{p\in\SP}E$ (in fact, strictly contained if Card$(E)\not=1$) whose topological closure is $\underline{E}$.
Moreover, for each prime $p$, $\pi_p(\widehat{E})=\pi_p(\underline{E})=E_p$. In particular, 
$$\Int(E,\Z)=\Int_{\Q}(\widehat{E},\hZ)=\Int_{\Q}(\underline{E},\hZ)$$

We know that locally, for a subset $E\subseteq\Z_{(p)}$, the polynomial closure of $E$ in $\Z_p$ coincides with its topological closure in $\Z_p$ (Proposition \ref{topclospolclos} \& Theorem \ref{containmentoverrings}). The global situation can be different: as the next example shows, in general $\widehat{E}$ can be strictly contained in $\underline{E}$. By definition, an element $\underline{\alpha}\in\underline{E}$ has the property that, for each finite set of primes $\{p_1,\ldots,p_k\}$ and finite set of non-negative integers $\{k_1,\ldots,k_s\}$, there exist $a_i\in E$ such that $a_i\equiv\alpha_{p_i}\pmod{p_i^{k_i}}$, for $i=1,\ldots,s$. In order for $\underline{\alpha}$ to belong to $\widehat{E}$, there should exist $a\in E$ which is a simultaneous solution of all the previous congruences.

\begin{Ex}
Let $E = \Z \setminus\{ -7 + 8\cdot 9 k \mid k\in\Z \}$. It is easy to see that $E$ is dense in $\Z_p$ for each prime $p$, so the polynomial closure of $E$ in $\hZ$ is equal to $\hZ$. However, since there is no $a\in E$ such that the following congruences are satisfied:
$$a \equiv 1 \pmod 8,\;\; a \equiv 2 \pmod 9$$
it follows that $\widehat{E}\subsetneq\underline{E}$.
\end{Ex}

\end{Rem}
\medskip
\begin{Ex}
Let us consider the ring $\Int(\Z)$. Since $\Int(\Z)=\bigcap_{p\in\mathbb{P}}\Int(\Z_{(p)})$ , by (\ref{IntZp}) we have
$$\Int(\Z)=\bigcap_{p\in\mathbb{P}}\Int_{\Q}(\Z_p)=\Int_{\Q}(\widehat{\Z},\widehat{\Z})\doteqdot\Int_{\Q}(\widehat{\Z})$$
Note the analogy of the previous equation with (\ref{IntZp}).
\end{Ex}

\begin{Cor}
For each $p\in\mathbb{P}$, let $R(p)$ be a polynomial overring of $\Int_{\Q}(\Z_{(p)})$. Then, there exists a polynomial overring $R$ of $\Int(\Z)$ such that $R_{(p)}=R(p)$ for each $p\in\mathbb{P}$. 
\end{Cor}
\begin{proof}
By Theorem \ref{containmentoverrings}, the choice, for each $p\in\mathbb{P}$, of an overring $R(p)$ of $\Int(\Z_{(p)})$ corresponds to a closed subset $F_p=Z_p(R(p))\subseteq\Z_p$. Moreover, $R(p)=\Int_{\Q}(F_p,\Z_p)$. Let $\underline{F}=\prod_{p\in\mathbb{P}}F_p\subseteq\hZ$ and let 
$$R=\Int_{\Q}(\underline{F},\widehat{\Z})=\bigcap_{p\in\SP}R(p)\,.$$
We claim that $R$ is the desired polynomial overring of $\Int(\Z)$. Since $R_{(p)}$ and $R(p)$ are elements of the family $\mathcal{R}_p$, by Theorem \ref{containmentoverrings}, it is sufficient to show that $Z_p(R_{(p)})=F_p$. Proposition \ref{unitaryvaloverrings} and (\ref{ZpRZpRp}) allow us to conclude.
\end{proof}


\begin{Rem}\label{ideles}
With our interpretation in terms of finite adeles, we may formulate Theorem~\ref{vqunder} in another way. Let $R$ be a polynomial overring of $\Int(\Z)$. Let $R_{\mathcal P, \underline{Z}_R}$ be a representation of $R$ as an intersection of valuation overrings (Proposition \ref{representationsglobal}). Theorem~\ref{vqunder} says that, for every $q\in\mathcal P,$ $V_q$ is superfluous if and only if there exists $\underline{\alpha}\in\underline{Z}_R$ such that $q(\underline{\alpha})$ is not invertible in $\mathcal A_f(\Q):$ the valuation domain $V_q$ is surperfluous in all representations of $R$ if and only if $q(\underline{Z}_R)\not\subseteq \mathcal I_f(\Q)$.

\end{Rem}

To end our study we show now under which conditions a polynomial overring $R$ of $\Int(\Z)$ is of the simple form $\Int(E,\Z)$ where $E$ is a subset of $\Z$. 

\begin{Cor}\label{simple}
A polynomial overring $R$ of $\Int(\Z)$ is a ring of integer-valued polynomials on a subset of $\Z$ if and only if, for each prime $p,$ the subset $E=\cap_p(Z_p(R)\cap\Z)$ is dense in $Z_p(R)$ for the $p$-adic topology. If this condition holds, then $R=\Int(E,\Z)$.
\end{Cor}

\begin{proof}
Clearly, $E=\{a\in\Z\mid R(a)\subseteq \Z\}$ and, if $R$ is a ring of integer-valued polynomials on a subset of $\Z,$ the subset $E$ is convenient. Moreover, the equality $R=\Int(E,\Z)$ holds if and only if both rings have the same localizations at each prime $p$. For every $p$, $\Int(E,\Z)_{(p)}=\Int_{\Q}(E,\Z_p)$ and $R_{(p)}=\Int_{\Q}(Z_p(R),\Z_p)$ by Proposition~\ref{RpIntval}. Thus, by Proposition~\ref{topclospolclos}, both localizations are equal if and only if $E$ is dense in $Z_p(R)$.
\end{proof}

\begin{Ex}
For each $p,$ let us consider the following closed subset of $\Z_p:$ $F_p=\{p\}\cup (\Z_p\setminus p\Z_p)$. Let $\underline{F}=\prod_pF_p\subseteq\hZ$ and $R=\Int_{\Q}(\underline{F},\hZ).$ Does there exist $E\subseteq\Z$ such that $R=\Int(E,\Z)?$ Yes, $R=\Int(\SP,\Z)$ since, for each $p$, the topological closure of $\SP$ in $\Z_p$ is $F_p.$ Actually, the subset $E$ suggested in Corollary~\ref{simple} is $\SP\cup\{\pm 1\}$, namely the polynomial closure of $\SP$ in $\Z$ (about $\Int(\SP,\Z)$ see~\cite{CCS}). 
\end{Ex}


\end{document}